\documentclass[11pt]{amsart}

\usepackage{amssymb}

\usepackage{graphicx}

\usepackage{enumerate}

\newtheorem{theorem}{Theorem}[section]
\newtheorem{lemma}[theorem]{Lemma}
\newtheorem{proposition}[theorem]{Proposition}

\theoremstyle{definition}
\newtheorem{definition}[theorem]{Definition}

\theoremstyle{remark}

\newcommand{\TryPackage}[3]{\IfFileExists{#1.sty}{\usepackage{#1}#2}{#3}}
\TryPackage{mathrsfs}{\renewcommand{\mathcal}{\mathscr}}{%
     \TryPackage{eucal}{}{}}

\newcommand{\ZZ}{{\mathbb Z}}

\newcommand{\CC}{{\mathbb C}}

\newcommand{\SLC}{{SL(2, {\mathbb C})}}

\begin{document}

\title
{State invariants of two-bridge knots}



\author{Cynthia L. Curtis}
\address{Department of Mathematics \& Statistics, The College of New Jersey, Ewing, NJ 08628}
\curraddr{}
\email{ccurtis@tcnj.edu}
\thanks{}

\author{Vincent Longo}
\address{Department of Mathematics, University of Nebraksa, Lincoln, NE 68588}
\curraddr{}
\email{vincent.longo@unl.edu}
\thanks{}


\dedicatory{}

\date{\today}
\begin{abstract}
In this paper, we consider generalizations of the Alexander polynomial and signature of 2-bridge knots by considering the Gordon-Litherland bilinear forms associated to essential state surfaces of the 2-bridge knots. We show that the resulting invariants are well-defined and explore properties of these invariants. Finally we realize the boundary slopes of the essential surfaces as a difference of signatures of the knot.

\end{abstract}

\maketitle


\section{Introduction} 
In \cite{GL}, Gordon and Litherland studied a bilinear form $GL$ 
associated to a spanning surface of a knot, identifying a correction term enabling them to use this form to compute the signature of the knot. We consider a family of matrices $V$ with $V_{GL} = V+V^T$ the matrix of $GL$ relative to a certain basis.  We use these matrices to build invariants generalizing the Alexander polynomial and the signature of the knot.

In order to obtain invariants, we restrict our attention to 2-bridge knots. Further, the spanning surfaces we consider are the essential spanning surfaces of the knot. These surfaces are of particular interest. Any 2-bridge knot has finitely many such surfaces, and these are well-understood up to isotopy due to the work of Hatcher and Thurston in \cite{HT}. For each such surface we define a {\it state polynomial} given by $det(V-tV^T)$ and a {\it state signature} $\sigma(V+V^T)$ and show that both quantities are independent of the choices involved in their definitions. 
We obtain a finite family of polynomials and a finite set of signatures for each 2-bridge knot.

The boundary slopes of these surfaces have become increasingly important computationally, for example in the computation  of Culler-Gordon-Luecke-Shalen semi-norms (see \cite{CGLS}, \cite{BZ01},\cite{C01} and \cite{O}), $\SLC$-Casson invariants (see \cite{C01},  \cite{BC06}, \cite{BC08}, and \cite{BC12}), and $A$-polynomials (see \cite{CCGLS}, \cite{BZ01} and \cite{BC12}). A motivating observation is that the Hatcher-Thurston formula for computing the boundary slope of an essential surface $S$ with boundary a 2-bridge knot is $2(M-M_0)$, where $M$ is a quantity computed from $S$ and $M_0$ is an analogous quantity computed from a  Seifert surface for the knot. In fact $M_0$ is just the signature of the knot $K$, and our motivating problem was to realize $M$ as a signature of a matrix defined from $S$. We conclude our paper by realizing the boundary slope of an essential spanning surface of a 2-bridge knot as the difference between the state signature associated to $S$ and the signature of the knot.

We note that our proofs of well-definiton of the invariants rely on the simplicity of the essential spanning surfaces of 2-bridge knots. 
An interesting question is whether an analogous construction continues to give polynomial invariants for other knots using their essential spanning surfaces. 

The paper is outlined as follows: in Section 2.1 we recall the definition of the Gordon-Litherland bilinear form and the definition of the Seifert matrix $V$ which we will generalize. In Section 2.2 we restrict our attention to 2-bridge knots, define the appropriate generalization of Seifert matrices, define the invariants, and show that they are well-defined. We investigate properties of the state polynomials in Section 3.1 and of the state signatures in Section 3.2.

\section{Definitions of the Invariants}
\subsection{The Gordon-Litherland form for a spanning surface of a knot} \label{GLsect}

We begin by introducing the bilinear form of Gordon and Litherland associated to a spanning surface following \cite{GL}. Let $K$ be a knot in $S^3$, and let $S$ be a spanning surface for $K$; that is, a surface with $\partial S = K$. Let $F$ be a closed tubular neighborhood of $S$, and view this as the total space of an $I$-bundle over $S$. Let  $\tilde{S}$ be the corresponding $\partial I$-bundle over $S$. 
Note that $\tilde{S}$ is the orientable double cover of $S$ if $S$ is non-orientable and is the trivial double cover if $S$ is orientable. 

Since 
$S$ and $\partial \tilde{S}$ are disjoint, the linking pairing $lk$ defines a bilinear form
$$lk:H_1(S)\times H_1(\tilde{S}) \rightarrow \ZZ.$$ Letting $\tau$ denote the transfer homomorphism $\tau: H_1(S)\rightarrow H_1(\tilde{S})$ and precomposing with $id \times \tau$ we obtain the Gordon Litherland pairing
$$GL: H_1(S)\times H_1(S) \rightarrow \ZZ$$ given by $GL(\alpha,\beta) = lk(\alpha, \tau(\beta))$.

If S is isotoped to a single disk with bands attached, then we may choose a set of curves $x_1,x_2,\ldots,x_k$  on $S$ given by the cores of the bands together with arcs in the disk joining the two ends of the core of each band. Orient each curve $x_i$ arbitrarily. The classes of the curves $x_i$ give a basis for $H_1(S)$. Relative to this basis we obtain a symmetric matrix $V_{GL}$ for the Gordon Litherland pairing. 

Note that if $S$ is orientable, then the matrix $V_{GL}$ is equal to $V+V^T$ for a certain Seifert matrix $V$ for $K$. The matrix $V$ is obtained directly from the linking pairing on $H_1(S)$ with respect to the chosen basis together with a choice of orientation on $S$. Specifically, denote the positive push-off of $x_i$ by $x^*_i$. We define the entries of $V$ by setting $v_{ij} = lk(x_i,x^*_j)$. It is easy to check that $V_{GL} = V+V^T$.

We remark that any Seifert matrix $V$ for a nontrivial knot is not symmetric. The asymmetry in $V$ arises precisely in the off-diagonal entries $v_{ij}\neq v_{ji}$ whenever the oriented intersection number of $x_i$ and  $x_j$ is $\pm 1$. For Seifert surfaces, the intersection paring for the curves $x_i$ remains consistent under band slides, and the Alexander polynomials obtained from the associated Seifert matrices $V$ and $V'$ agree. 

This is not quite the case for non-orientable surfaces. To see this, consider the non-orientable surfaces $S$ and $S'$ with boundary the $5_2$ knot shown in Figure \ref{52surfaces}. These surfaces are isotopic, as $S'$ is obtained from $S$ by sliding the right foot of the band on the left in $S$ counterclockwise along the right band in $S$. However the curves in a basis coming from the cores of the bands in the first picture intersect, whereas the curves in the basis given by the cores of the second band do not. Defining diagonal terms by $\frac{1}{2}lk(x_i,\tau(x_i))$, we obtain Seifert-like matrices $V = 
\left[
\begin{array}{cc}
1  & 0     \\
 1 &-3/2
\end{array}
\right]
$ and $V' = 
\left[
\begin{array}{cc}
1/2  &   1   \\
 1  &   -3/2   
\end{array}
\right]
$. Both $V_{GL}$ and $V'_{GL}$ are matrices for $GL$ relative to the corresponding bases. However $V$ is asymmetric while $V'$ is symmetric, and the polynomials $det(V-tV^T)$ and $det(V'-tV'^T)$ do not agree.

\begin{figure}
\begin{center}
\leavevmode\hbox{}
\includegraphics[scale=0.1]{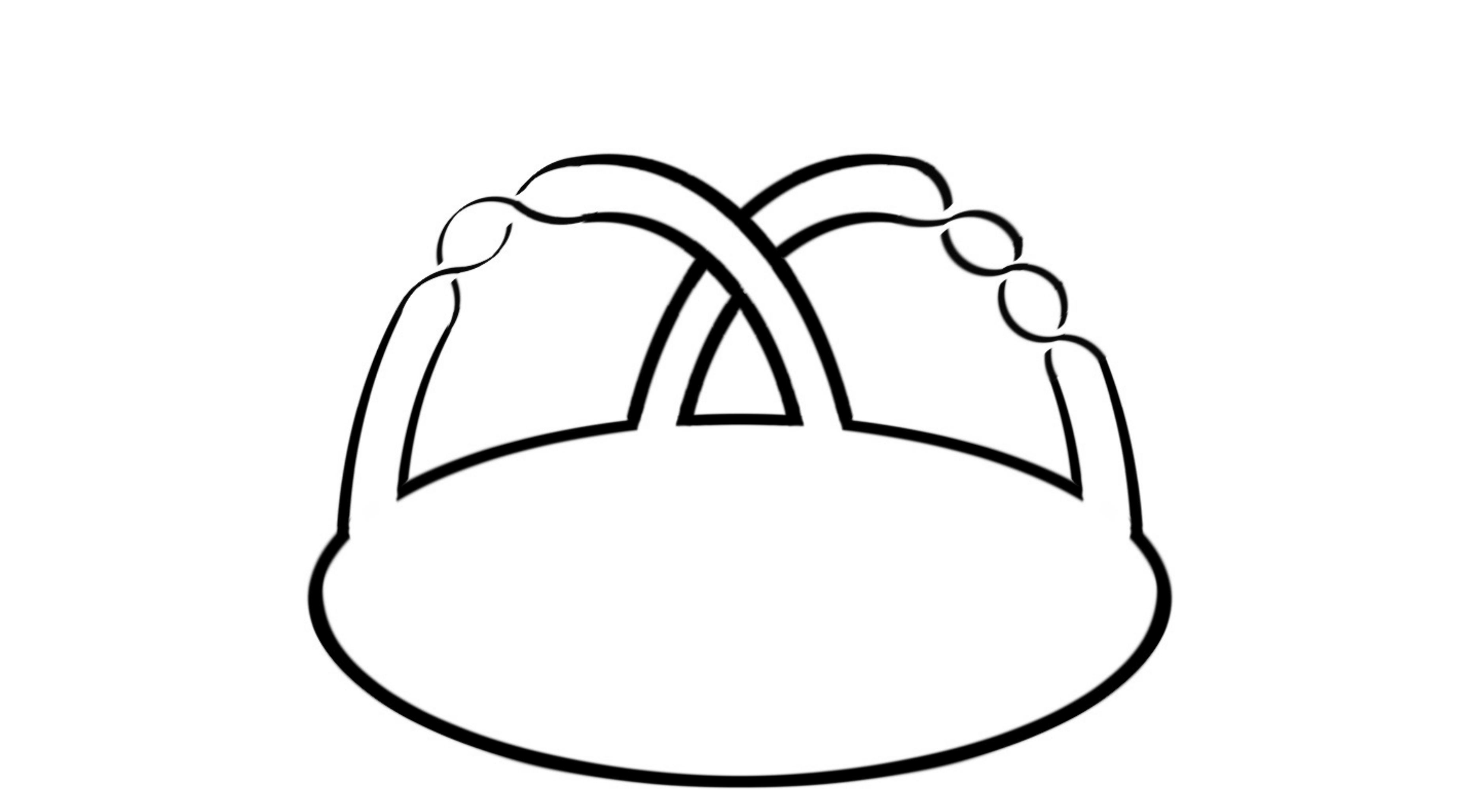}
\includegraphics[scale=0.1]{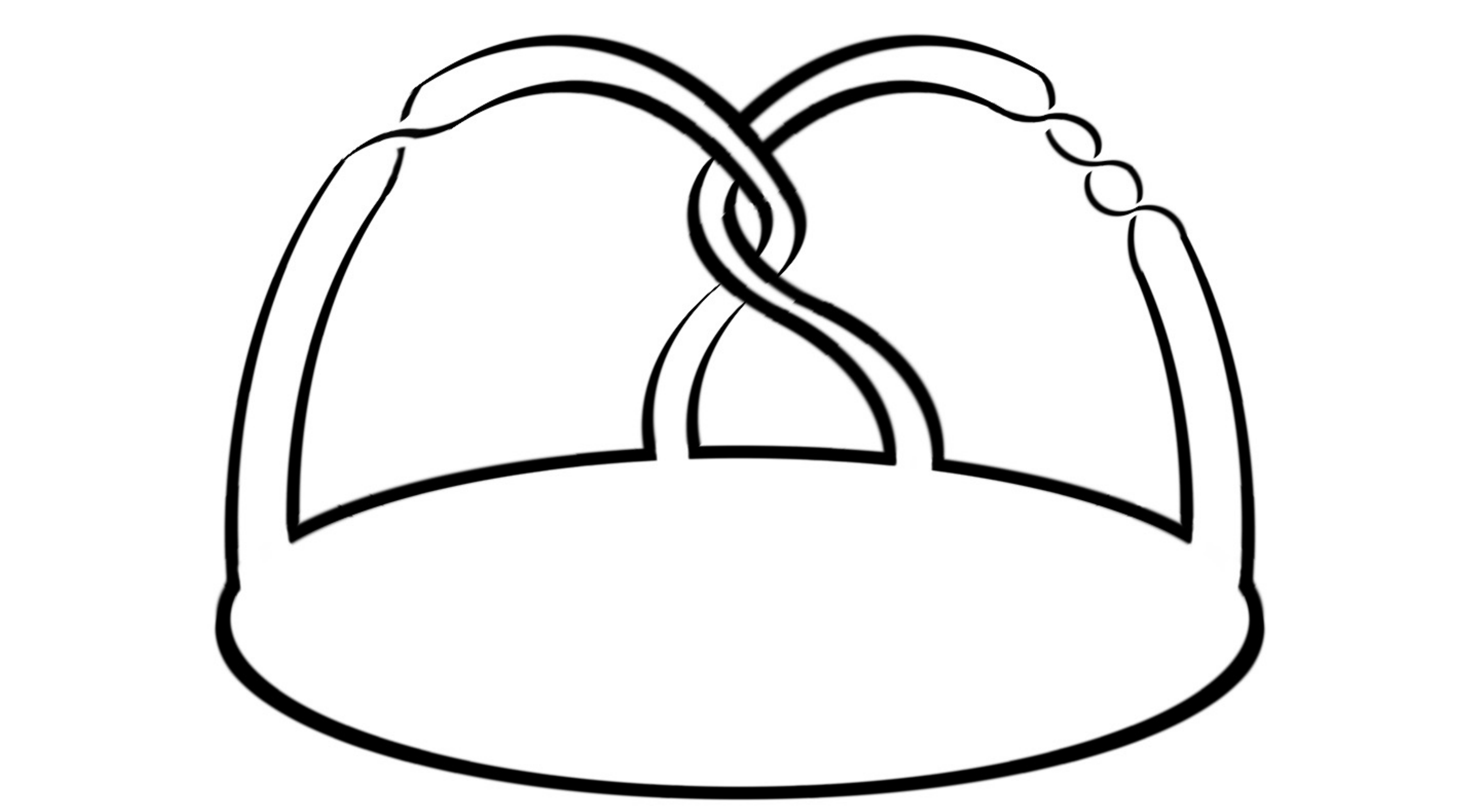}
\caption{The surfaces $S$ and $S'$ corresponding to $[2,3]$} 
\label{52surfaces}
\end{center}
\end{figure}


Thus we must be very careful in generalizing the construction of the Seifert matrix to non-orientable surfaces, so that we obtain a sensible and consistent choice of asymmetry in the resulting matrices. We solve this problem for 2-bridge knots below.




 \subsection{Definitions of the invariants}\label{refs}
 
 For the remainder of our paper, we assume $K$ is a 2-bridge knot, and let $M = S^3 - N(K)$, where $N(K)$ is an open tubular neighborhood of $K$. Let $K = K(\alpha , \beta)$ be the standard 2-bridge notation for $K$, where $\alpha$ is the determinant of $K$.
 
 A surface $\Sigma$ in $M$ is said to be \textit{incompressible} if for any disk $D \subset M$ with $D \cap \Sigma = \partial D$, there exists a disk  $D' \subset \Sigma$, with $\partial D' = \partial D$.  A surface $\Sigma$ is $\partial${\em - incompressible} if for each disk $D \subset M$ with $D \cap \Sigma = \partial_+D$ and $D \cap \partial M  = \partial_-D$ there is a disk $D' \subset \Sigma$ with $\partial_+ D' =\partial_+ D$ and $\partial_- D'  \subset \partial \Sigma$. A surface $\Sigma \subset M$ is \textit{essential} if it is both incompressible and $\partial$-incompressible. 
 
 By \cite{HT} we know that each essential spanning surface $S$ is isotopic to at least one of the following surfaces obtained by plumbing together half-twisted bands: Specifically, we choose half-twisted bands
 corresponding to a continued fraction expansion 
 $$\frac{\beta}{\alpha} = r + \frac{1}{n_1 + \frac{1}{n_2 + \ldots +\frac{1}{n_k}}}$$
with $|n_i|\geq 2$ for each $i$. Henceforth we denote this continued fraction expansion by $[n_1,n_2, \ldots, n_k].$ We arrange these bands vertically as shown in Firgure \ref{HTsurface}, attach parallel untwisted bands as shown, attach the top (level 0) and bottom (level $k$) shaded horizontal rectangular disks shown, and at each level (1,2,\ldots k-1) between the bands we attach either the shaded horizontal rectangular disk shown (the ``inner'' disk) or its complement in the sphere which is the one-point compactification of the horizontal plane (the ``outer'' disk).

\begin{figure}
\begin{center}
\leavevmode\hbox{}
\includegraphics[scale=0.33]{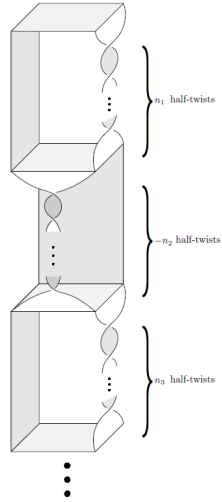}
\caption{The surface corresponding to $[n_1,n_2,...,n_k]$} 
\label{HTsurface}
\end{center}
\end{figure}

For 
such a plumbed surface 
with corresponding continued fraction expansion $[n_1,n_2,...,n_k]$, there is a collection of curves $x_1,x_2,...,x_k$ well-defined up to isotopy of the curve system on the surface, where $x_i$ is essentially the core of the $i^{th}$ ``box'', consisting of a vertical arc forming the core of the $i^{th}$ twisted band, a vertical arc following the opposite untwisted band, and arcs joining the ends of the vertical arcs in the level disks (whether inner or outer) in levels $i-1$ and $i$. 
Note that $x_i \cap x_{i+1}$ can be chosen to be transverse, consisting of a single point in the horizontal disk in
level $i$ (either inner or outer), and $x_i \cap x_j = \emptyset $ if $j \notin \{i-1,i, i+1\}$. Fix such a collection of curves, and orient these arbitrarily. We now generalize the construction of the Seifert matrix as follows:

Fix an isotopy of $S$ to a plumbed surface with corresponding continued fraction expansion $[n_1,n_2,\ldots,n_k]$, and (abusing notation) let $x_1,x_2,...,x_k$ also denote the curve system on $S$ corresponding to the chosen curve system $x_1,x_2,...,x_k$ on the plumbed surface. 
For each $i$ and $j$ such that $j \neq i$ and $x_i\cap x_j \neq \emptyset$ let $D_{ij}$ be a small disk neighborhood of $x_i \cap x_j$ in $S$, chosen so that no two distinct neighborhoods $D_{ij}$ intersect, and let $N_{ij}$ be an arbitrarily chosen normal vector to $D_{ij}$. Denote by $x^{i*}_j$ a curve obtained from $x_j$ by an isotopy supported in a tubular neighborhood of $D_{ij}$ pushing $x_j$ in the direction of $N_{ij}$ if $x_i$ and $x_j$ meet, and let $x^{i*}_j = x_j$ if $x_i$ and $x_j$ are disjoint. 

\begin{definition}
A {\it state matrix} $V_S$ for an essential spanning surface $S$ is a matrix $V_S$ with entries
$$v_{ij} = \left\{  \begin{array}{ll}
							lk(x_i,x^{i*}_j) & \mbox{if } i\neq j \\
							\frac{1}{2} lk(x_i, \tau(x_i)) &  \mbox{if } i=j. \end{array} \right. $$ 

\end{definition}
Note that if $S$ is nonorientable, then the diagonal entries in $V_S$ will be half-integers rather than integers. Also note that there are several matrices $V_S$ for any surface $S$ due to the choices involved in the definition. As in the orientable case, it is clear that $V_{GL} = V_S + V^T_S$ is a matrix for $GL$ for any state matrix $V_S$.

We now use the state matrices associated to these surfaces to define our invariants. Let $S$ be an essential spanning surface for $K$, and let $V_S$ be an associated state matrix. We define generalizations $\Delta_S(t)$ and $\sigma_S$ of the Alexander polynomial and the signature of the knot using $S$. As for the Alexander polynomial, we define two polynomials to be equivalent if they differ by a unit $\pm t^k$ in the ring of integral Laurent polynomials, where $k \in \ZZ$.

\begin{definition}
The {\it state polynomial} $\Delta_{S}(t)$  is the equivalence class of $det(V_S - t V^{T}_{S})$.
\end{definition}
 We will frequently abuse notation and identify $\Delta_S(t)$ by a representative of the equivalence class, as is standard for Alexander polynomials.
 
\begin{definition}
The {\it state signature} $\sigma_{S}$ is the signature of $V_S + V^{T}_{S}$.
\end{definition}

\begin{theorem}\label{invariance}
Let $K$ be a 2-bridge knot, and let $S$ be an essential spanning surface of $K$. The state polynomial $\Delta_{S}$ and the state signature $\sigma_{S}$ depend only on the isotopy classes of $K$ and $S$. Thus the collection of state polynomials and the collection of state signatures are well-defined invariants of the knot $K$.
\end{theorem}

The theorem, and indeed the construction of the polynomials, is somewhat simpler if we view our surfaces as given by a single disk with bands attached rather than as plumbings of twisted bands. Therefore before proving the theorem we first show the following:
\begin{lemma}\label{discwithbands}
Let $K$ be a 2-bridge knot, and let $S$ be an essential spanning surface of $K$ corresponding to the continued fraction expansion $[n_1,n_2,\ldots,n_k]$ which is a plumbing of twisted bands as described above, with corresponding curve system $x_1,x_2,\ldots, x_k$. Then $S$ together with the curve system $x_1,x_2,\ldots x_k$ is isotopic to a single disc with $n$ unknotted bands with curve system given by the cores of the bands, where the $i^{th}$ band has $n_i$ half-twists if $i$ is odd and $-n_i$ half-twists if $i$ is even. The bands are arranged so that the feet of the $i^{th}$ and $(i+1)^{st}$ bands are alternating along the disk, 
and so that the feet of the $i^{th}$ and $j^{th}$ bands do not alternate along the disk if $i$ and $j$ are not adjacent integers. If all horizontal disks are inner disks, then the $2i^{th}$ band crosses in front of both the $(2i-1)^{st}$ band and the $(2i+1)^{st}$ band. For each level $j$ at which $S$ uses the outer rather than the inner disk, the crossing of bands $j$ and $j+1$ is the reverse of that for the surface with all inner disks. 
\end{lemma}
This realization of $S$ is shown in Figure \ref{flat73surface}  for the knot $7_3$ where $S$ is the surface corresponding to the continued fraction expansion $[4, -2,2,-2]$ all of whose horizontal disks are inner disks.

\begin{figure}
\begin{center}
\leavevmode\hbox{}
\includegraphics[scale=0.33]{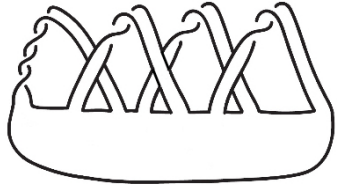}
\caption{The branched surface corresponding to $[4,-2,2,-2]$, with all inner disks, realized as a single disk with bands} 
\label{flat73surface}
\end{center}
\end{figure}

\begin{proof}
View the surface shown in Figure \ref{HTsurface} as a disk with bands attached, where the disk is the union of the shaded vertical (untwisted) rectangles together with the horizontal (inner or outer) disks. Flatten this surface into the plane of the even-numbered vertical rectangles.
\end{proof}

\begin{proof}
We now prove Theorem \ref{invariance}. 
We made a number of choices in defining $V_S$, and we must show that both 
$\Delta_S(t)$ and $\sigma_S$ are 
is independent of these choices. Specifically, we have we have chosen the curve system $x_1,x_2,\ldots, x_k$ to be the cores of the handles of $S$ after first isotoping $S$ to one of the standard surfaces described above; we must investigate the choice involved in this isotopy.
In addition, we have numbered and oriented the curves $x_i$ arbitrarily, and we have chosen normal vectors for the discs $D_{ij}$ arbitrarily.  

We first show independence of the numbering and orientation of the curves $x_i$. Renumbering the curves will alter $V_S$ by reordering the rows and corresponding columns of $V_S$ and therefore also the rows and corresponding columns of $V^T$ and $V_S-tV^T_S$.  This will leave $\sigma_S$ and $\Delta_S(t)$ unchanged. Reversing the orientation of a curve $x_i$ will alter $V_S$ by multiplying the $i^{th}$ row and $i^{th}$ column by $-1$ to obtain a new state matrix $V'_S$. Then $V_S$ and $V'_S$ differ by congruence via a matrix $U$, where $U$ is a diagonal matrix with $i^{th}$ diagonal entry $-1$ and all other diagonal entries $1$: $V_S' = UV_SU^T$. Therefore  $det(V'_S - t V'^{T}_S) = det(V_S - t V^T_S)$, so the state polynomial is unaffected by these choices. 
Further since $V_S$ and $V'_S$ are conjugate the signs of the eigenvalues of $V_S+V_S^T$ and $V'_S+V'^{T}_S$ agree, so $\sigma_S$ is also independent of these choices.


Next we consider the dependence of $\Delta_S(t)$ and $\sigma_S$ on the choices of normal vectors $N_{ij}$.
Note that there is a choice of numbering and orientation for the curves $x_i$ and a choice of normal vectors $N_{ij}$ so that the resulting matrix $V = V_S$ is given by 
\begin{eqnarray}\label{standardV}
V=\left[ \begin{array}{cccccl}
\frac{n_1}{2} & 0 &  &   &  &  \\
1 & -\frac{n_2}{2} & 0  &  &  & \text{\huge0} \\
0 & 1 & \frac{n_3}{2} & 0 &  & \\
 & 0 & 1 & -\frac{n_4}{2} & \ddots & \\
 &  & \ddots & \ddots & \ddots & \ \ \ \ 0\\
 & \text{\huge0}   &  & 0 & 1 & (-1)^{k+1}\frac{n_k}{2}\\
\end{array}\right].
\end{eqnarray}
If we change the orientation of a single normal vector $N_{i,i+1}$ we obtain the matrix $V'$ which is identical to $V$ except at entry $v'_{i,i+1}$, which is 1 rather than 0 and entry $v'_{i+1,i}$, which is  0 rather than 1.

Now $V+V^T$ and $V'+V'^T$ are identical, so $\sigma_S$ is unchanged by the change in $N_{i,i+1}$. To see that $\Delta_S(t)$ is unchanged by the reversal of $N_{i,i+1}$, note that 
\begin{eqnarray}\label{V-tVT}
V-tV^T=\left[ \begin{array}{cccccccc}
m_1 & -t & 0 &  &   &\text{\huge0} \\
1 & -m_2 & -t  & 0 &  & \\
0 & 1 & m_3 & -t & \ddots & \\
 & 0 & 1 & -m_4 & \ddots & 0\\
 \text{\huge0}&  &\ddots & \ddots & \ddots & -t\\
&   &  & 0 & 1 & (-1)^{k+1}m_k\\
\end{array}\right],
\end{eqnarray}
where $m_j = \frac{n_j}{2} - t \frac{n_j}{2}$, while 
$$V' -t V'^T=\left[ \begin{array}{ccccccccc}

m_1 & -t & &  &  &  &  & &\\

1 & -m_2 & -t  & &  & \text{\huge0}& & & \\

 & 1 & m_3 & \ddots &  &  &  & &\\

 & & \ddots & \ddots & -t &  &  & & \\

 &  &  & 1 & (-1)^{i+1}m_{i} & 1 &  & & \\

 &  &  &  & -t & (-1)^{i+2}m_{i+1} & -t & & \\

 &  &\text{\huge0}  &  &  & 1 & \ddots & \ddots &\\
 
& & & & & & \ddots & \ddots & -t\\

 &  &  &  &  &  & & 1 & (-1)^{k+1}m_k\\

\end{array}\right].$$
We see that $V-tV^T$ is obtained from $V'-tV'^T$ by multiplying each of rows $i+1,i+2,\ldots,k$ of $V'-tV'^T$ by $-1/t$ and each of columns $i+1,i+2,\ldots,k$ of $V'-tV'^T$ by $-t$. Then $det(V - t V^{T}) = det(V'-tV'^T)$, and hence $\Delta_S(t)$ is unchanged by reversing $N_{i,i+1}$.

It remains to be shown that the choice of representation of $S$ as a plumbing of twisted bands given by a continued fraction expansion does not affect the resulting polynomials or signatures. By Theorem 1(b) of \cite{HT}, $S$ must be isotopic to such a surface, and by Theorem 1(d) of \cite{HT}, the continued fraction expansion giving rise to $S$ is unique (as we have required $|n_i| \geq 2$ for each $i$). 
If $S$ is isotopic to a given surface via two different isotopies, giving rise to distinct curve systems $x_1,x_2,\ldots,x_k$ and $x'_1,x'_2,\ldots,x'_k$ on $S$ each corresponding to the curve system $x_1,x_2,\ldots, x_k$ on the plumbed surface, then the various state matrices defined by these curve systems with will agree up to choices of normal vectors since the linking numbers $lk(x_i,x_j^{i*})$ for and $lk(x'_i,x_j^{'i*})$ must each agree with the linking numbers $lk(x_i,x_j^{i*})$ for the curves on the plumbed surface up to choice of normal vector and similarly $lk(x_i,\tau(x_i))$ and $lk(x'_i,\tau(x'_i))$ on $S$ must both agree with $lk(x_i,\tau(x_i))$ on the plumbed surface.

However by Theorem 1(e) of \cite{HT} there is any isotopy relation among the various essential spanning surfaces corresponding to a given continued fraction expansion. Specifically, if $n_i = \pm 2$ for one or more $n_i$, then the choice of inner versus outer disk at one or more levels in the representation of $S$ is not uniquely determined by $S$. 

We show that this ambiguity in the representation of $S$ does not impact the resulting state polynomial or state signature. In fact, any two essential spanning surfaces for the knot corresponding to the same continued fraction expansion give the same state polynomial and state signature. To see this, recall from the proof of Lemma 2.5 that replacing the $i^{th}$ inner horizontal disk with the $i^{th}$ outer horizontal disk has the effect of reversing the crossing between bands $i$ and $i+1$ in the representation of $S$ as a disk with bands attached given by Lemma 2.5. But  reversing the crossing between bands $i$ and $i+1$ has the same effect on the associated state matrix as replacing $N_{i,i+1}$ with its negative and reversing the orientation of $x_{i+1}, x_{i+2},\ldots,x_k$. Therefore by the earlier parts of this proof the state polynomial and the state signature are not affected by a reversal of the crossing between bands. Hence the polynomial and the signature depend only on the isotopy class of $S$.

\end{proof}
 
\section{Properties of the invariants}
We continue to let $K$ be a 2-bridge knot with essential spanning surface $S$. Suppose the continued fraction expansion corresponding to $S$ is $[n_1,n_2,\ldots,n_k]$.
 
 \subsection{Properties of state polynomials and relation to genera}
 
\begin{theorem}\label{polyprops}
The state polynomial $\Delta_{S}$ has the following properties:
\begin{enumerate}[i]
	\item The polynomial is symmetric; that is, $\Delta_{S}(t^{-1}) \sim \Delta_{S}(t)$.  \\
	\item $\Delta_{S}(1) = \left\{  \begin{array}{ll}
							1 & \mbox{if $k$ is even} \\
							0 & \mbox{if $k$ is odd.} \end{array} \right. $ \\
	\item $|\Delta_{S}(-1)| = det(K)$.\\
	\item When normalized to be a polynomial with lowest term a constant term, $\Delta_{S}$ is a polynomial of degree $k$ with leading coefficient $\frac{1}{2^k} n_1 n_2 \ldots n_k$.
	\end{enumerate}
\end{theorem}

\begin{proof}
Isotope $S$ to be the disk with bands given by Lemma \ref{discwithbands}. Choose the curves $x_i$ and orient the normal vectors $N_{ij}$ so that the resulting matrix $V$ is given by matrix 
(\ref{standardV}) above.

To prove property $i$, note 
$$\begin{array}{lcl}
\Delta_{S}(t) &=& det(V-tV^T)\\
&=& det[(V-tV^T)^T]\\
&=& det(-tV+V^T)\\
&=& det[-t(V-t^{-1}V^T)]\\
&=& \pm t^k det(V-t^{-1}V^T)\\
&=& \pm t^k \Delta_S(t^{-1}).\\
\end{array}$$
Thus the two polynomials are equivalent.

To prove property $ii$, observe that
$$\Delta_S(1) = det(V-V^T) = \left[ \begin{array}{cccccccc}
0 & -1 & 0 &  &   &\text{\huge0} \\
1 & 0 & -1  & 0 &  & \\
0 & 1 & 0 & -1 & \ddots & \\
 & 0 & 1 & 0 & \ddots & 0\\
 \text{\huge0}&  &\ddots & \ddots & \ddots & -1\\
&   &  & 0 & 1 & 0\\
\end{array}\right].$$
If $k=1$ we have $\Delta_S(1) = det[0] = 0$, and if $k=2$ we have $\Delta_S(1) = det \left[ \begin{array}{cc} 0&-1\\1&0\end{array}\right] =1.$ Further, if $k>1$, it is easy to check by cofactor expansion along the first row and then by cofactor expansion along the first column of the $k-1 \times k-1$ submatrix that $det(V-V^T) = det M$, where $M$ is the $k-2 \times k-2$  matrix of the same form as $V-V^T$. The claim follows by induction.

For property $iii$, note that $\Delta_S(-1) = det(V+V^T)$,
and recall that $V+V^T = V_{GL}$. The result is immediate since $V_{GL}$ is a Goeritz matrix by Section 2 of \cite{GL}.















Finally we turn to property $iv$. Note that if $\Delta_S$ is normalized as indicated, then the constant term is given by $\Delta_S(0) = det(V)$ for any state matrix $V$. We have seen that we may take $V$ to be lower triangular, so $det(V) = \pm \frac{1}{2^k} n_1n_2\cdots n_k$. By symmetry this is also the coefficient of $t^k$.

\end{proof}

Note that if $S$ is orientable, then $\Delta_{S}$ is the Alexander polynomial of $K$. In this case the properties listed are well-known. Note in particular that if $S$ is orientable then $k$ is even.

We define the \textit{genus} of a non-orientable surface $\Sigma$, like that of an orientable surface, to be $\frac{1 - \chi(\Sigma)}{2}$, or equivalently to be half the number of cross-cap summands in the surface.  Note that the genus of $S$ is $g(S) = k/2$, so that in general the degree of the state polynomial $\Delta_S(t)$ is $2g(S)$. We examine the relationship between the various genera and hence the various degrees of the state polynomials for a knot $K$.

Recall that the \textit{genus of K} is the minimum genus among all Seifert surfaces of $K$. We denote this by $g(K)$. In our case this is the genus of the orientable essential surface $S$, corresponding to the unique continued fraction expansion all of whose quotients $n_1,n_2,\ldots,n_k$ are even. We show that $g(K)$ and the genera $g(S)$ for nonorientable essential spanning surfaces $S$ are essentially independent, thereby establishing that the degrees of the Alexander polynomial of $K$ and the other state polynomials of $K$ are independent. 

We remark that the \textit{nonorientable genus of K} is defined to be the minimum genus among all nonorientable spanning surfaces of $K$.  In Section 3 of \cite{AK} the authors show that the nonorientable genus of $K$ is equal to the minimum genus among the essential spanning surfaces for $K$ if this genus is realized by a nonorientable surface or is equal to $g(K)+1/2$ otherwise. This is also shown explicitly, and the nonorientable genera are computed explicitly, for 2-bridge knots in \cite{HiTe}. 
Note that if the nonorientable genus is realized by an essential spanning surface of $K$ then the nonorientable genus is equal to half the degree of the corresponding state polynomial. On the other hand, if no nonorientable essential spanning surface realizes the nonorientable genus of knot, then the nonorientable genus is realized by a surface obtained by adding a cross-cap to the minimal genus Seifert surface. In this case the nonorientable genus is not equal to half the degree of any state polynomial of $K$.

We prove the following

\begin{proposition} The degrees of the state polynomials of a 2-bridge knot $K$ are independent of the degree of the Alexander polynomial of $K$. Specifically:
\begin{enumerate}[i]
	\item There exist 2-bridge knots with linear state polynomials and  Alexander polynomials of arbitrarily high degree. \\
	\item There exist 2-bridge knots with quadratic Alexander polynomials and with state polynomials of arbitrarily high degree.
\end{enumerate}
\end{proposition}

\begin{proof}
To prove the theorem, we identify knots with continued fraction expansions of the appropriate lengths. See for example the algorithm of Proposition 3.3 of \cite{CFLM} for identifying the continued fraction expansions for a knot $K$.

For claim $i$, consider the $(2,m)$ torus knot. This has a continued fraction expansion $[m]$, which corresponds to a surface $S$ with genus 1/2 and linear state polynomial $\Delta_S$. The Seifert surface for the $(2,m)$ torus knot corresponds to the continued fraction expansion $[-2,2,-2,2,\ldots,-2,2]$ of length $m-1$. Then the degree of the Alexander polynomial is $2g(K) = m-1$.

For claim $ii$, consider the 2-bridge knot with continued fraction expansion $[2i,2j]$. This continued fraction expansion yields the Seifert surface of the knot, so the knot is genus one with a quadratic Alexander polynomial. The other two essential spanning surfaces of the knot have continued fraction expansions $[-2,2,-2,\ldots,-2,2j + 1]$ and $[2i + 1,-2,2,-2,\ldots,-2]$ of length $2i$ and $2j$, respectively. These surfaces are of genera $i$ and $j$, respectively, and the corresponding state polynomials have degrees $2i$ and $2j.$

\end{proof}





\subsection{Properties of state signatures and relation to boundary slopes}
We next consider properties of the state signatures of the 2-bridge knot $K$. We continue to let $S$ be an essential spanning surface for $K$ corresponding to the continued fraction expansion $[n_1,n_2,\ldots,n_k]$. We begin by explaining how to compute the state signature $\sigma_S$ in terms of the continued fraction expansion. Specifically, let $N^+$ denote the number of entries in the list $n_1,n_2,\ldots n_k$ whose signs agree with the alternating pattern  $+,-,+,-,\ldots$, and let $N^-$ denote the number of entries in the list $n_1,n_2,\ldots n_k$ whose signs do not agree with the alternating pattern $+,-,+,-,\ldots$. 

\begin{proposition}
The state signature $\sigma_S$ is given by $\sigma_S = N^+ - N^-.$
\end{proposition}

\begin{proof}
We are interested in the signature of $V+V^T$.
Note that 
\begin{eqnarray*}
V+V^T = \left[ \begin{array}{ccccccc}

n_1 & 1 & 0 & &  \text{\huge0}& \\

1 & -n_2 & 1 & 0 & &\\

0 & 1 & n_3 & 1& \ddots &\\

& 0 & 1 & \ddots & \ddots &0\\

\text{\huge0}& & \ddots& \ddots & (-1)^kn_{k-1} &1 \\

  & & & 0 & 1 & (-1)^{k+1}n_k \\

\end{array}\right],\end{eqnarray*}
which is row equivalent to
 
\begin{eqnarray*}\left[ \begin{array}{ccccccc}

n_1 & 1 &  & & \text{\huge0}\\

0 & -\big(n_2+\frac{1}{n_1}\big) & 1 &  & \\

 & 0 & \bigg(n_3+\frac{1}{n_2+\frac{1}{n_1}}\bigg) & \ddots& \\

&  & \ddots & \ddots & 1 \\

\text{\huge0}& & & 0 & (-1)^{k+1}\bigg(n_k+\frac{1}{n_{k-1}+\ldots+\frac{1}{n_1}}\bigg)  \\

\end{array}\right].\end{eqnarray*}
Therefore the signs of the eigenvalues of $V+V^T$ agree with the signs of the numbers in the list $n_1, -(n_2+\frac{1}{n_1}), n_3 +\frac{1}{n_2+\frac{1}{n_1}}, \ldots, (-1)^k (n_k+\frac{1}{n_{k-1}+\ldots+\frac{1}{n_1}})$. Since $|n_i| > \left| \frac{1}{n_{i-1}+\ldots+\frac{1}{n_1}}\right|$, we see these signs agree with the signs of the numbers in the list $n_1,- n_2, n_3, \ldots, (-1)^k n_k$. Thus the number of positive eigenvalues is precisely $N^+$, and the number of negative eigenvalues is $N^-$.
\end{proof}

This proposition makes the computation of the state signatures simple, and allows us to easily verify some properties of state signatures.

\begin{proposition}
The state signatures $\sigma_S$ for a given knot $K$ have the following properties:
\begin{enumerate}[i]
	\item $|\sigma_S| \leq 2g(S)$.
	\item The width of the set of state signatures for a 2-bridge knot may be arbitrarily large.
	\item The numbers $\sigma(K)$ and $\sigma_S$ are independent; that is, either may be arbitrarily large while the other is 0.
\end{enumerate}
\end{proposition}
Here note that the width of the set of state signatures is the difference between the largest and the smallest state signature.

\begin{proof}
For the first claim, note that the number of eigenvalues of $V+V^T$ is $2g(S)$, so this is an upper bound for the absolute value of the signature.

For the remaining claims, first consider the knot $K(4m+1,2m)$ with continued fraction expansion $[2,2m]$ giving the Seifert surface and with continued fraction expansion $[3,-2,2,-2,\ldots,-2]$ of length $2m$ giving a second essential spanning surface $S$. We see that for the Seifert surface $N^+ = N^- = 1$, so $\sigma(K) = 0$. However for $S$ we see that $N^+=2m$ and $N^-=0$, so $\sigma_S = 2m$.

In contrast, the knot  $K(6\ell+1,2\ell)$ where $\ell$ is any integer has a continued fraction $[4,-2,2,-2,\ldots,-2]$  of length $2\ell$ yielding the Seifert surface and a continued fraction expansion $[3, 2\ell]$ yielding an essential spanning surface $S$. We see that for the Seifert surface $N^+=2 \ell$ and $N^- = 0$, whereas for $S$ we have $N^+=N^-=1$. Hence $\sigma(K) = 2 \ell$ while $\sigma_S = 0$.

\end{proof}

Finally, we return to our motivating problem, realizing the boundary slopes of $K$ as a difference of signatures of $K$:

\begin{theorem}
Let $S$ be an essential spanning surface of $K$. Then the boundary slope of $S$ is given by
$2(\sigma_S - \sigma(K))$.
\end{theorem}

\begin{proof}
By Proposition 2 of \cite{HT} we know the boundary slope of the essential spanning surface $S$ corresponding to the continued fraction expansion $[n_1,n_2,\ldots,n_k]$ is given by $2(N^+ - N^-) - 2(N^+_0 - N^-_0)$, where $N^+_0$ and $N^-_0$ are $N^+$ and $N^-$ for the all-even continued fraction expansion yielding the essential Seifert surface of $K$. (Here note that our sign conventions do not agree with those of \cite{HT}; this is their result restated with our sign conventions.) But $2(N^+ - N^-) = 2\sigma_S$, and $2(N^+_0 - N^-_0) = \sigma(K)$.
\end{proof}

\bigskip
\noindent
{\it Acknowledgements.} The authors thank William Franczak for contributing Figure \ref{HTsurface}. In addition we thank the referee for his or her insightful suggestions.

\end{document}